\documentclass{amsart}

\usepackage[normalem]{ulem}

\usepackage{lineno}

\usepackage{soul}

\usepackage{amssymb, amsmath, amsthm, hyperref, euscript, enumerate, color}
\usepackage[dvipsnames]{xcolor}

\usepackage{amscd}

\usepackage{bbm}

\usepackage{enumitem}

\usepackage[english]{babel}

\numberwithin{equation}{section}

\theoremstyle{plain}

\newtheorem{theorem}{Theorem}

\newtheorem{proposition}{Proposition}

\newtheorem{remark}{Remark}

\newtheorem{lemma}{Lemma}

\newtheorem{thm}{Theorem}

\newenvironment{hproof}{%
  \proof}{\endproof}

\theoremstyle{definition}

\newtheorem{definition}{Definition}

\newcommand{\R}{\mathbb{R}}

\newcommand{\Z}{\mathbb{Z}}

\newcommand{\N}{\mathbb{N}}

\newcommand{\C}{\mathbb{C}}

\newcommand{\Hom}{\textnormal{Hom}}

\newcommand{\id}{\textnormal{id}}

\newcommand{\Tor}{\textnormal{Tor}}

\newcommand{\Spin}{\mathrm{Spin}}

\DeclareMathOperator{\Or}{O}
\DeclareMathOperator{\Ker}{ker}

\DeclareMathOperator{\Ks}{KS}
\DeclareMathOperator{\Sw}{SW}

\makeatletter
\@namedef{subjclassname@2020}{%
  \textup{2020} Mathematics Subject Classification}
\makeatother

\begin{document}

\author{Rafael Torres}

\title[DIFF Knotted/TOP unknotted nullhomotopic 2-spheres]{Smoothly knotted and topologically unknotted nullhomologous surfaces in 4-manifolds.}

\address{Scuola Internazionale Superiori di Studi Avanzati (SISSA)\\ Via Bonomea 265\\34136\\Trieste\\Italy}

\email{rtorres@sissa.it}


\maketitle

\emph{Abstract}: We point out that recent constructions of inequivalent smooth structures yield a manufacturing procedure of infinite sets of topologically isotopic and pairwise smoothly inequivalent nullhomologous 2-spheres and 2-tori inside a myriad of 4-manifolds. These surfaces bound a locally flat embedded handlebody but not a smoothly embedded one, and include the first examples of nullhomotopic 2-spheres in closed 4-manifolds to display such behavior. We exhibit an example of a locally flat non-smoothable embedding of a nullhomotopic 2-sphere in a closed smooth simply connected 4-manifold.

\section{Introduction and main result}\label{Introduction}

All manifolds and embeddings in this paper are in the smooth category unless it is otherwise specified. Two surfaces $\Sigma_1$ and $\Sigma_2$ embedded in a 4-manifold $X$ are said to be equivalent if there is a diffeomorphism of pairs\begin{equation}\label{Diffeo Pairs Intro}(X, \Sigma_1)\rightarrow (X, \Sigma_2)\end{equation}and inequivalent if there is no such map. The surface $\Sigma_i$ is said to be smoothly knotted if it does not bound an embedded handlebody $H\subset X$ up to isotopy. When the surfaces $\Sigma_1$ and $\Sigma_2$ are locally flat embedded in a topological 4-manifold $X$, we say that they are topologically isotopic if there is a homeomorphism of pairs as (\ref{Diffeo Pairs Intro}) that is isotopic to the identity map. Furthermore, we say that $\Sigma_i$ is topologically unknotted if it is topologically isotopic to a locally flat embedded surface $\Sigma\subset X$ that bounds a locally flat embedded handlebody $H\subset X$. Among the many interesting four-dimensional phenomena is the existence of embedded surfaces $\Sigma_i\hookrightarrow X$ that are inequivalent within a fixed topological isotopy class. We call such embeddings exotic embeddings. Fintushel-Stern introduced a surgery procedure on an embedded 2-torus that results in infinite sets of exotic embeddings of 2-tori with simply connected complement \cite{[FintushelSternS]}. Finashin \cite{[Finashin]}, Kim \cite{[Kim2]}, and Kim-Ruberman \cite{[KimRuberman], [KimRuberman2]} used a variation of Fintushel-Stern's surgery procedure to construct exotic embeddings of surfaces whose complements have cyclic fundamental groups. Examples of exotic embeddings of homologically essential 2-spheres have been constructed by Akbulut \cite{[AkbulutR]}, Auckly-Kim-Melvin-Ruberman \cite{[AuckleyKimMelvinRuberman]}, and Ruberman \cite{[Ruberman]}. In \cite[Theorem 1]{[HoffmanSunukjian]}, Hoffman-Sunukjian exhibited infinite sets of exotic embeddings of nullhomologous 2-tori that are topologically unknotted and smoothly knotted, which arise from Fintushel-Stern's knot surgery construction \cite{[FintushelSternK]}. 

While it is still unknown if there are exotic embeddings of topologically unknotted 2-spheres in the 4-sphere, we provide the first examples of smoothly knotted/topologically unknotted nullhomotopic 2-spheres inside closed 4-manifolds in this paper. 
  
\begin{thm}\label{Theorem A} Let $X$ be a closed symplectic 4-manifold that contains a 2-torus $T$, which is either symplectic or homologically essential and Lagrangian. Suppose that the normal bundle of $T$ is trivial, and that both $X$ and $X\setminus \nu(T)$ are simply connected.

$\bullet$ There is an infinite set of topologically isotopic 2-spheres\begin{equation}\label{Inequivalent Spheres}\{S_i: i \in \Z\}\end{equation}embedded in $X\#S^2\times S^2$ that satisfy $[S_i] = 0 \in H_2(X\#S^2\times S^2; \Z)$ and 2-knot group $\pi_1(X\#S^2\times S^2\backslash \nu(S_i)) = \Z$ for every $i\in \Z$, and which are pairwise inequivalent. 

Performing surgery to $X\#S^2\times S^2$ along each element of (\ref{Inequivalent Spheres}) yields an infinite set $\{X_n(1): n\in \Z\}$ of pairwise non-diffeomorphic 4-manifolds that are homeomorphic to $X\#S^2\times S^2\#S^1\times S^3$. 

$\bullet$ Suppose that $b_2(X)\geq |\sigma(X)| + 6$. There is a an infinite set of topologically isotopic 2-tori\begin{equation}\label{Inequivalent Tori}\{T_i': i \in \N\}\end{equation}embedded in $X$ that satisfy $[T_i'] = 0 \in H_2(X; \Z)$ and $\pi_1(X\backslash \nu(T_i')) = \Z$  for every $i\in \Z$, and which are pairwise inequivalent.

There is a torus surgery that can be performed to $X$ along each element of (\ref{Inequivalent Tori}), which yields an infinite set $\{X_n(1): n\in \Z\}$ of pairwise non-diffeomorphic 4-manifolds that are homeomorphic to $X\#S^2\times S^2\#S^1\times S^3$. 

\end{thm}

Each element of (\ref{Inequivalent Spheres}) and (\ref{Inequivalent Tori}) is topologically unknotted and smoothly knotted in $X\#S^2\times S^2$ and $X$, respectively; see Section \ref{Section Proof of Theorem A}. For instance, each element of (\ref{Inequivalent Spheres}) bounds a locally flat embedded 3-ball in $X\#S^2\times S^2$ up to isotopy, but no element bounds an embedded 3-ball. The first  part of Theorem \ref{Theorem A} was kindly suggested to us by Bob Gompf after an unsuccessful attempt to construct inequivalent smooth structures on 4-manifolds that lack an almost-complex structure \`a la Fintushel-Stern \cite[\S 5]{[FintushelStern1]}. Previously known examples of exotic embeddings require either for the 2-sphere to be homologically essential, for the nullhomologous surfaces to have genus equal to one or for the ambient 4-manifold to have non-empty boundary; see Akbulut \cite{[AkbulutR]}, Auckly-Kim-Melvin-Ruberman \cite{[AuckleyKimMelvinRuberman]}, Fintushel-Stern \cite{[FintushelSternS]}, Mark \cite{[Mark]}, Kim-Ruberman \cite{[KimRuberman]}, Hayden \cite{[Hayden]}, Hoffman-Sunukjian \cite{[HoffmanSunukjian]}, Juh\'asz-Miller-Zemke \cite{[JuhaszMillerZemke]}, Oba \cite{[Oba]}, Ruberman \cite{[Ruberman]}. The second part of Theorem \ref{Theorem A} is verbatim a result of Hoffman-Sunukjian \cite[Theorem 1.1]{[HoffmanSunukjian]}. Our contribution here is a different proof of their result, where the 2-tori are distinguished using an invariant introduced by Fintushel-Stern \cite[Section 2]{[FintushelSternTori]}.

Sunukjian provided conditions for the existence of a homeomorphism of pairs (\ref{Diffeo Pairs Intro}) for two locally flat embedded surfaces $\Sigma_i$ in a topological 4-manifold $X$ \cite[Section 7]{[Sunukjian]}. In the case the fundamental group of the complement $X\setminus \nu(\Sigma_i)$ is cyclic, Sunukjian obtained a criteria to determine when the surfaces are topologically isotopic \cite[Theorems 7.1 - 7.4]{[Sunukjian]}. Another contribution of this paper is a criteria to determine the existence of a homeomorphism of pairs for locally flat embedded 2-spheres in terms of the homeomorphism classes of the 4-manifolds obtained by applying surgery to them; see Section \ref{Section Spherical Modifications}. The definition of product framing in the following statement is given in Definition \ref{Definition Product Framing}.

\begin{thm}\label{Theorem C}Let $X_i$ be a closed connected orientable 4-manifold that contains an embedded 2-torus $T$ of self-intersection zero and fix a framing $\nu(T)\rightarrow T^2\times D^2$. Let $\gamma_i\subset T\subset X_i$ be an embedded loop whose homotopy class is a generator of the group $\pi_1X_i = \Z$, which is assumed to have the product framing for $i = 1, 2$. Let $X_{\gamma_i}$ be the closed connected simply connected 4-manifold that is obtained by carving out a tubular neighborhood $\nu(\gamma_i)$ of the loop $\gamma_i\subset X_i$ and gluing back in a copy of $D^2\times S^2$, and denote by $S_i:= \{0\}\times S^2\subset D^2\times S^2$ be the belt 2-sphere of the surgery. 

The following statements are equivalent

\begin{enumerate}
\item There is a homeomorphism\begin{equation}\label{Homeo 1}X_1\rightarrow X_2.\end{equation}

\item There is a homeomorphism of pairs\begin{equation}\label{Homeo 1'}f:(X_1, \gamma_1)\rightarrow (X_2, \gamma_2).\end{equation}


\item There is a homeomorphism of pairs\begin{equation}\label{Homeo 2'B}F:(X_{\gamma_1}, S_1)\rightarrow (X_{\gamma_2}, S_2).\end{equation} 
\end{enumerate}Furthermore, a homeomorphism (\ref{Homeo 2'B}) taken as a homeomorphism $F: X_{\gamma_1}\rightarrow X_{\gamma_2}$ is degree 1 normally bordant to the identity of $X_{\gamma_2}$.

\end{thm}

The choice of framing is relevant as Gordon's examples of topologically inequivalent 2-spheres inside $S^4$ in \cite{[Gordon]} indicate. The utility of Theorem \ref{Theorem C} arises from a codimension three  property of loops embedded in a 4-manifold: any two homotopic embeddings $\gamma_j\hookrightarrow X$ are isotopic \cite[Definition 1.1.5, Example 4.1.3]{[GompfStipsicz]}.  In the situation of Theorem \ref{Theorem A}, there is a diffeomorphism between  $X_{\gamma_1}$ and $X_{\gamma_2}$, and we regard the 2-spheres $S_1$ and $S_2$ as embedded in $X_{\gamma_2}$. Work of Perron \cite{[Perron]} and Quinn \cite{[Quinn]} allows us to conclude that $S_1$ is topologically isotopic to $S_2$ in this case. 

Hambleton-Teichner \cite{[HambletonTeichner]} and Friedl-Hambleton-Melvin-Teichner \cite{[FriedlHambletonMelvinTeichner]} constructed closed orientable non-smoothable topological 4-manifolds with  fundamental group $\Z$ that are not homotopy equivalent to the connected sum of a simply connected 4-manifold with $S^1\times S^3$. An interesting consequence of their work is the following result.

\begin{thm}\label{Theorem FHMT} There are locally flat embedded nullhomotopic 2-spheres\begin{equation}\label{Topol Emb}S,\Sigma\hookrightarrow 4\mathbb{CP}^2\end{equation}that are concordant and whose exteriors are not homotopy equivalent. 

While the embedding (\ref{Topol Emb}) can be taken to be smooth for $S$, the locally flat embedding of $\Sigma$ is not topologically isotopic to a smooth embedding. 

\end{thm}

Theorem \ref{Theorem FHMT} exhibits the first example of a locally flat embedded nullhomotopic 2-sphere in a closed 4-manifold that is topologically concordant but not topologically isotopic to a smooth embedding. Moreover, while Freedman-Quinn \cite[11.7A Theorem]{[FreedmanQuinn]} have shown that any 2-sphere inside $S^4$ with infinite cyclic 2-knot group is topologically unknotted, Theorem \ref{Theorem FHMT} points out that this unknotting phenomena need not occur for 2-spheres that are locally flat embedded in closed 4-manifolds with larger second Betti number.

The cut-and-paste techniques considered in this paper yield several choices of ambient manifolds as we sample in the following theorem.

\begin{thm}\label{Theorem D}Let $X\in \{S^2\times S^2, S^1\times S^3, S^2\times T^2\}$. There is an infinite set\begin{equation}\label{Tori Example B}\{T_i': i \in \Z\}\end{equation}of pairwise smoothly inequivalent 2-tori embedded in $X$ whose second homology class is $[T_i'] = 0 \in H_2(X; \Z)$. 

There is a an infinite set\begin{equation}\label{Knots Example B}\{S_i: i \in \Z\}\end{equation}of pairwise smoothly inequivalent 2-spheres embedded in $X\#S^2\times S^2$ whose second homology class is $[S_i] = 0 \in H_2(X\#S^2\times S^2; \Z)$. 

\end{thm}

Theorem \ref{Theorem D} places torus surgeries and surgeries along as additions to the list of constructions of examples of surfaces with interesting 2-knot group. The latter include twist spinning surfaces in Kim \cite{[Kim2]} and the principle of localized knotted of Hayden in\cite{[Hayden]}. Our proofs of Theorem \ref{Theorem A} and Theorem \ref{Theorem D} are driven by recent constructions of inequivalent smooth structures on closed simply connected 4-manifolds $X$ with small second Betti number due to Akhmedov-Park \cite{[AkhmedovPark1], [AkhmedovPark2]}, Akhmedov-Baykur-Park \cite{[AkhmedovBaykurPark]}, Baldridge-Kirk \cite{[BaldridgeKirk1], [BaldridgeKirk2]}, Fintushel-Stern \cite{[FintushelStern2], [FintushelStern3]}, Fintushel-Stern-Park \cite{[FintushelParkStern]}, Szab\'o \cite{[Szabo]}. These results are based on performing torus surgeries (see Section \ref{Section Torus Surgeries} for the definition) to a 4-manifold $Y$ with non-trivial fundamental group and non-trivial Seiberg-Witten invariants. Under the right conditions, the toughest step in all of these constructions is arguably the computation of the fundamental group. We observe that it is straight-forward to produce a simply connected 4-manifold if one is willing to apply to $Y$ either a torus surgery of multiplicity zero or a surgery along a loop whose homotopy class corresponds to a generator of the fundamental group; Section \ref{Section Torus Surgeries} and Section \ref{Section Spherical Modifications}. Performing either of these surgeries to $Y$ comes at the price of rendering useless the gauge theoretical invariants of the 4-manifold produced for the purposes of discerning smooth structures. Moreover, in a myriad of instances, the 4-manifold obtained after the surgeries is diffeomorphic to $X$ and no inequivalent smooth structure is unveiled. Nevertheless, our efforts are not entirely moot and we do obtain infinitely many pairwise smoothly inequivalent nullhomologous surfaces inside $X$ and $X\# S^2\times S^2$.

We organized the paper in the following way. The results that are involved in the proofs of our theorems are contained in Section \ref{Section Tools}. A description of the cut-and-paste constructions of 4-manifolds that are used is given in Section \ref{Section Torus Surgeries} and in Section \ref{Section Spherical Modifications}. The latter contains a proof of Theorem \ref{Theorem C}. Specific choices of surgeries as well as several diffeomorphisms that are useful for our proof of Theorem \ref{Theorem A} are given in Section \ref{Section 2-tori in the 4-torus}. Proofs of Theorem \ref{Theorem A}, Theorem \ref{Theorem FHMT} and Theorem \ref{Theorem D} are given in Section \ref{Section Proofs Main Results}.

\subsection{Acknowledgements:} We are indebted to Bob Gompf for useful e-mail correspondence, which motivated the writing of this note. We thank John Etnyre and Bob Gompf for patiently pointing out the need to rectify the proof of Lemma 1 in an earlier version of the note. We thank Maggie Miller and Danny Ruberman, and an anonymous referee for their time and their suggestions, which helped us improve the note. Je remercie Baptiste Morin et Sol Avino pour l'aide dans la traduction du r\'esum\'e.

\section{Tools and technology}\label{Section Tools}

\subsection{Torus surgeries }\label{Section Torus Surgeries} Let $T\hookrightarrow X$ be an embedded 2-torus inside a closed smooth 4-manifold $X$. Suppose that $[T]^2 = 0$ so that its tubular neighborhood $\nu(T)$ is diffeomorphic to the thick torus $T^2\times D^2$. A framing of $T$ is a diffeomorphism\begin{equation}\label{Framing Fixed}\phi: \nu T\rightarrow T^2\times D^2\end{equation}such that $\phi^{-1}(T^2\times \{0\}) = T$. Consider the projection $\pi: T^2\times D^2\rightarrow T^2$ and denote by $\{x, y\}$ homotopy classes of loops that carry the generators of the fundamental group $\pi_1(T) = \Z x \oplus \Z y$. The push-offs of $x$ and $y$ are loops on $\partial \nu(T)$ given by\begin{center}$m = \phi^{-1}(\pi(x)\times \{d\})$ and $l = \phi^{-1}(\pi(y)\times \{d\})$\end{center} for $d\in \partial D^2$. The framing of $T$ is encoded by the pair $\{m, l\}$, with $m$ and $l$ being  homologous to $x$ and $y$, respectively, in $\nu(T)$. The meridian $\mu_T$ of the 2-torus $T$ in the complement $X\backslash \nu(T)$ is a fixed curve in the isotopy class of $\{t\}\times \partial D^2 \subset \partial \nu(T)$, and the group $H_1(\partial \nu(T); \Z) = \Z^3$ is generated by\begin{equation}\label{Loops generating torus H1}\{m, l, \mu_T\}.\end{equation}The surgery curve $\gamma\subset \partial \nu(T)$ is the push off of a primitive loop in $T$. 


\begin{definition}\label{Definition Torus Surgery}Let $T\subset X$ be an embedded 2-torus with a fixed framing (\ref{Framing Fixed}) and a surgery curve $\gamma\subset T$. A $(p, q, r)$-torus surgery on $T$ along $\gamma$ is the cut-and-paste operation that removes a copy of $\nu(T)$ from $X$ and glues back a copy of $T^2\times D^2$ to produce a closed smooth 4-manifold\begin{equation}\label{Torus Surgery}X_{T, \gamma}(p, q, r):= (X\backslash \nu(T)) \cup_{\varphi} (T^2\times D^2)\end{equation}where the common boundaries are identified by a diffeomorphism\begin{equation}\varphi: T^2\times \partial D^2 \rightarrow \partial(X\backslash \nu(T))\end{equation} that satisfies \begin{equation}\label{Induced Map in Homology}\varphi_{\ast}([\{pt\}\times \partial D^2]) = p[m] + q[l] + r [\mu_T]\end{equation} in $H_1(\partial (X\backslash \nu(T)); \Z)$ for integers $p, q$, and $r$. The integer $r$ is known as the multiplicity of the $(p,q, r)$-torus surgery.

\end{definition}

The cut-and-paste construction of Definition \ref{Definition Torus Surgery} is also known as a smooth logarithmic transformation in the literature \cite[p. 83]{[GompfStipsicz]}, \cite[Section 2.2]{[BaykurSunukjian]}. A $(p, q, r)$-torus surgery (\ref{Torus Surgery}) can be undone. The 4-manifold $X$ is recovered by applying a $(p', q', r')$-torus surgery on the core torus\begin{equation}\label{Core Torus}T'\subset X_{T, \gamma}(p, q, r)\end{equation}for $T':= T^2\times \{0\}\subset T^2\times D^2$ in (\ref{Torus Surgery}) along a surgery curve $\gamma'\subset X_{T, \gamma}(p, q, r)$. This is a local operation on the 4-manifold in the sense that\begin{equation}\label{Same Tori Complements}X_{T, \gamma}(p, q, r)\setminus \nu(T') = X\setminus \nu(T).\end{equation}If the 2-torus $T$ and the surgery curve $\gamma$ are essential in $H_2(X)$ and $H_1(X\setminus \nu(T))$, respectively, then the core 2-torus (\ref{Core Torus}) is nullhomologous \cite[Section 3]{[FintushelStern3]}. A nullhomologous 2-torus has a canonical framing known as the nullhomologous framing. This is the framing of $T$ such that $\phi^{-1}_{\ast}([T\times \{d\}])\in \Ker i_\ast$ for the inclusion map $i: \partial \nu(T)\rightarrow X\setminus \nu(T)$. In particular, the loops $\{m, l\}$ are nullhomologous in the complement.

\subsection{Surgeries along loops and 2-spheres}\label{Section Spherical Modifications} Let $\gamma \subset X$ be a loop embedded in a closed smooth orientable 4-manifold $X$. The tubular neighborhood $\nu(\gamma)$ is diffeomorphic to $S^1\times D^3$. There are two choices of framing for $\gamma$ given that $\pi_1(\Or(3)) = \Z/2$ \cite[Section 4.1]{[GompfStipsicz]}. We pin down a specific choice of framing in the following definition. 

\begin{definition}\label{Definition Product Framing}Let $T\subset X$ be a 2-torus of self-intersection zero with a fixed framing $\nu(T)\rightarrow T^2\times D^2$, and such that $\gamma\subset T\subset X$. The product framing on the loop $\gamma\subset X$ is the framing induced by the $\nu(T) = T^2\times D^2$ product structure.
\end{definition}

Now that a framing has been pinned down, we can unambiguously define the 4-manifold obtained from $X$ by doing surgery along $\gamma$ in the following definition.

\begin{definition}\label{Definition Spherical Modification} \cite[Definition 5.2.1]{[GompfStipsicz]}. A surgery on $X$ along a closed simple loop that is contained in an embedded $T^2\times D^2\subset X$ is the cut-and-paste procedure that removes a copy of $\nu(\gamma)$ from $X$ and caps the boundary off with a copy of $D^2\times S^2$ to produce a closed smooth 4-manifold\begin{equation}\label{Spherical Modification}X_{\gamma}:= (X\backslash \nu(\gamma)) \cup_{\id} (D^2\times S^2)\end{equation}using a diffeomorphism of the boundary that is isotopic to the identity map. 
\end{definition}

Notice that the surgery of Definition \ref{Definition Spherical Modification} can be reversed. The belt 2-sphere of the surgery (\ref{Spherical Modification}) is an embedded 2-sphere $S\hookrightarrow X_{\gamma}$ of self-intersection zero, and we recover the original 4-manifold as\begin{equation}\label{Reverse Spherical Modification}X = (X_{\gamma}\backslash \nu(S)) \cup_{\id} (S^1\times D^3),\end{equation}where the gluing diffeomorphism of the boundary is isotopic to the identity map.

In the main constructions in this paper, the homotopy class of the belt 2-sphere is zero as we state in the following lemma. 

\begin{lemma}Let $X$ be a closed oriented 4-manifold and let $\gamma\subset X$ be an embedded loop whose homotopy class generates the fundamental group $\pi_1X = \Z$. The belt 2-sphere $S\subset X_\gamma$ of the surgery (\ref{Spherical Modification}) is nullhomotopic. 

\end{lemma}

By employing results of Freedman-Quinn \cite[Section 9.3]{[FreedmanQuinn]} on the existence and uniqueness of a tubular neighborhood of a submanifold of a topological 4-manifold in the locally flat category allow, it is possible to obtain a natural extension of these results to the locally flat category. 


We now recall the statement of Theorem \ref{Theorem C} from the introduction and prove it.

\begin{theorem}\label{Theorem Homeomorphism Spheres}Let $X_i$ be a closed connected orientable 4-manifold that contains an embedded 2-torus $T$ of self-intersection zero and fix a framing $\nu(T)\rightarrow T^2\times D^2$. Let $\gamma_i\subset T\subset X_i$ be an embedded loop whose homotopy class is a generator of the group $\pi_1X_i = \Z$, and which is assumed to have the product framing for $i = 1, 2$ as in Definition \ref{Definition Product Framing}. Let $X_{\gamma_i}$ be the closed connected simply connected 4-manifold that is obtained by performing surgery along $\gamma_i\subset X_i$.

The belt 2-sphere $S_i:= \{0\}\times S^2\subset D^2\times S^2$ of the surgery (\ref{Spherical Modification}) is nullhomotopic. 

The following statements are equivalent

\begin{enumerate}
\item There is a homeomorphism\begin{equation}\label{Homeo 1}X_1\rightarrow X_2.\end{equation}

\item There is a homeomorphism of pairs\begin{equation}\label{Homeo 1'}f:(X_1, \gamma_1)\rightarrow (X_2, \gamma_2).\end{equation}


\item There is a homeomorphism of pairs\begin{equation}\label{Homeo 2'}F:(X_{\gamma_1}, S_1)\rightarrow (X_{\gamma_2}, S_2).\end{equation}

\end{enumerate}

The homeomorphism of pairs (\ref{Homeo 2'}) taken as a homeomorphism\begin{equation}\label{Homeomorphism 2'}F: X_{\gamma_1} \rightarrow X_{\gamma_2}\end{equation} is normally bordant to the identity of $X_{\gamma_2}$.

\end{theorem}

We follow closely work of Kim-Ruberman \cite[Appendix, Proofs of Theorem 1.2 and Theorem 1.3]{[KimRuberman]} in the proof of the result.

\begin{proof} The equivalence $(1) \Leftrightarrow (2)$ follows from the codimension three property of the submanifolds $\gamma_i\subset X_i$, which says that they are isotopic if and only if they are homotopic  \cite[Example 4.1.3]{[GompfStipsicz]}. The homeomorphism (\ref{Homeo 1}) yields a homeomorphism of pairs $(X_1, \gamma_1)\rightarrow (X_2, \gamma_2)$:  since $[\gamma_1] = [\gamma_2]\in \pi_1(X_i)$, then the loops are isotopic.


We argue the equivalence $(2)\Leftrightarrow (3)$. Assuming (2), we have the diagram\begin{equation}\label{Diagram 1 Proof}
\begin{CD}
X_1\setminus \nu(\gamma_1) @>\id_1>> X_{{\gamma_1}}\setminus \nu(S_1)\\
@VVf_0V            @VVF_0V\\
X_2\setminus \nu(\gamma_2)@>\id_2>>  X_{{\gamma_2}}\setminus \nu(S_2).
\end{CD}
\end{equation}

Both horizontal homeomorphisms exist due to the local nature of the cut-and-paste construction (\ref{Spherical Modification}) and the vertical homeomorphism $f_0$ on the left-side of the diagram exists by the existence of the homeomorphism (\ref{Homeo 1'}). We obtain a homeomorphism $F_0: X_{{\gamma_1}}\setminus \nu(S_1) \rightarrow X_{{\gamma_2}}\setminus \nu(S_2)$ that extends to a homeomorphism of pairs $F: (X_{{\gamma_1}}, S_1)\rightarrow (X_{{\gamma_2}}, S_2)$ due to the choice of product framing on the loops and the gluing diffeomorphism in (\ref{Spherical Modification}) being isotopic to the identity map. We conclude that $(2)\Rightarrow (3)$. The converse implication follows similarly.

We need to show that the homeomorphism (\ref{Homeomorphism 2'}) has trivial normal invariant. The normal invariant $n(F)$ of the homeomorphism (\ref{Homeomorphism 2'}) lies in the group $[X_{\gamma}, G/TOP] \cong H^2(X_{\gamma}; \Z/2)\oplus H^4(X_{\gamma}; \Z)$ \cite[Section 2]{[Quinn]}, \cite{[CochranHabegger]}. As argued by Kim-Ruberman \cite[p. 5872]{[KimRuberman]}, the component of $n(F)$ that lies in the fourth cohomology group vanishes since $F$ is a homeomorphism; their argument works for homotopy equivalences. Denote by $S(F)$ the component of the normal invariant $n(F)$ that lies in $H^2(X_{\gamma}; \Z/2) = \Hom(H_2(X_{\gamma}; \Z/2), \Z/2)$. The class $S(F)$ is determined by its evaluation $\langle S(F), [\Sigma]\rangle$ on surfaces $\Sigma\subset X_{\gamma}$, and Kim-Ruberman provide a method to evaluate it in terms of the Arf invariant \cite[Proof of Lemma 2.4]{[KimRuberman]} using Cochran-Habegger \cite[Theorem 5.1]{[CochranHabegger]}. Perturbing $F$, we get a normal map $F^{-1}\Sigma\rightarrow \Sigma$ with surgery obstruction $\langle S(F), [\Sigma]\rangle$. The evaluation depends only on the homology class of the surface, but not on its representative and it is given by the Arf invariant \cite[p. 5873]{[KimRuberman]}. The invariant vanishes since the 2-spheres are nullhomologous.
\end{proof}

\begin{remark}\label{Remark Infinite Cyclic}Work of Freedman-Quinn \cite[Theorem 10.7A]{[FreedmanQuinn]} and Conway-Powell \cite[Theorem 5.17]{[ConwayPowell]} can be used to show that the implication $(3)\Rightarrow (1)$ of Theorem \ref{Theorem Homeomorphism Spheres} holds. Denote by $\Lambda := \Z[t^\pm]$ the group ring of Laurent polynomials \cite{[ConwayPowell]}. We reconstruct the $\Lambda$-valued intersection forms $\lambda_{X_1}$ and $\lambda_{X_2}$ from Diagram (\ref{Diagram 1 Proof}) and show that they are isometric building on an argument due to Conway-Powell \cite[proof of Lemma 5.16]{[ConwayPowell]}. Diagram (\ref{Diagram 1 Proof}) implies that there is an isometry\begin{equation}\lambda_{X_1\backslash \nu(\gamma_1)}\rightarrow \lambda_{X_2\backslash \nu(\gamma_2)}\end{equation} of the $\Lambda$-valued intersection forms $\lambda_{X_1\backslash \nu(\gamma_1)}$ and $\lambda_{X_2\backslash \nu(\gamma_2)}$. It also holds that\begin{equation}H_1(\partial (X_1\backslash \nu(\gamma_i)); \Lambda) = H_1(S^1\times S^2; \Lambda) = H_1(\R \times S^2; \Z) = 0.\end{equation} At this point we use the universal coefficient spectral sequence with second page $E^2_{p, q} = \Tor^{\Lambda}_p(H_q(X_i\backslash \nu(\gamma_i); \Lambda), \Z)$ that converges to $H_\ast(X_i\backslash \nu(\gamma_i))$ to conclude  that,\begin{equation}\Tor_2^{\Lambda}(H_0(X_i\backslash \nu(\gamma_i); \Lambda), \Z) = \Tor^{\Lambda}_2(\Z, \Z) = H_2(\Z, \Z) = 0.\end{equation}Therefore,\begin{equation}H_2(X_i\backslash \nu(\gamma_i); \Z) = H_2(X_i\backslash \nu(\gamma_i); \Lambda)\otimes_{\Lambda} \Z\end{equation}are isomorphic for $i= 1, 2$. Since $X_i$ is the union of $X_i\backslash \nu(\gamma_i))$ and $S^1\times D^3$, we conclude that the $\Lambda$-valued intersection forms $\lambda_{X_1}$ and $\lambda_{X_2}$ are isometric by \cite[Proposition 3.9]{[ConwayPowell]}. A result of Freedman-Quinn \cite[Theorem 10.7A]{[FreedmanQuinn]} implies that $X_1$ is homeomorphic to $X_2$. 
\end{remark}

\subsection{Three pairs of geometrically dual 2-tori in the 4-torus, some useful diffeomorphisms and a homeomorphism}\label{Section 2-tori in the 4-torus} We proceed to build several 4-manifolds from the 4-torus that will be useful in determining the diffeomorphism types in the proof of Theorem \ref{Theorem A}. Define\begin{equation}\label{4-torus}T^4 = x\times y\times a\times b,\end{equation}where $x, y, a, b$ is each a copy of $S^1$ and their homotopy classes generate the free abelian group\begin{equation}\pi_1(T^4) = \Z^4 = \Z x\oplus \Z y \oplus \Z a \oplus \Z b.\end{equation} We will follow the notation of Baldridge-Kirk \cite[Section 2.2]{[BaldridgeKirk2]} and denote by $X$ and $Y$ parallel push offs of $x$ and $y$, respectively, in $x\times y\times \{p_a\}\times \{p_b\}$ for points $p_a\in a$ and $p_b\in b$. Analogously, we let $A_1, A_2$ and $A_3$, and $B$ be parallel push offs of $a$ and $b$, respectively in $\{p_x\}\times \{p_y\}\times a\times b$ for points $p_x\in x$ and $p_y\in y$; see \cite[Figure 1]{[BaldridgeKirk2]}. 

We define three disjoint submanifolds of (\ref{4-torus}). The first two are\begin{equation}\label{Torus 1} T_1:= X\times \{p_y\}\times A_1\times \{p_b\},\end{equation}and\begin{equation}\label{Torus 2} T_2:= \{p_x\}\times Y\times A_2\times  \{p_b\},\end{equation}
which are chosen exactly as it is done by Baldridge-Kirk in \cite[Section 2.2]{[BaldridgeKirk2]}. The third submanifold is 
\begin{equation}\label{Torus 3} T_3:= \{p_x\}\times \{p_y\} \times A_3 \times B.\end{equation}

The 2-tori (\ref{Torus 1}) (\ref{Torus 2}) (\ref{Torus 3}), together with their geometrically dual 2-tori, generate the second homology group $H_2(T^4; \Z) = \Z^6$. We specify the framings for the 2-tori that were discussed in Section \ref{Section Torus Surgeries} by equipping the 4-torus with the product symplectic structure $(x\times y)\times (a\times b)$. The 2-torus (\ref{Torus 3}) is a symplectic submanifold, while the 2-tori (\ref{Torus 1}) and (\ref{Torus 2}) are Lagrangian with respect to this choice of symplectic structure. The $(p, q, r)$-torus surgeries on (\ref{Torus 1}) and/or (\ref{Torus 2}) are performed with respect to the Lagrangian framing \cite{[HoLi]}. Performing torus surgeries and surgeries along loops to the 4-torus yields the following 4-manifolds. We follow the notation in \cite[Section 2.1]{[BaldridgeKirk2]} and the terminology introduced in Definition \ref{Definition Torus Surgery} and Definition \ref{Definition Product Framing}.

\begin{lemma}\label{Lemma Diffeomorphism 1}Let $n\in \Z$. Let $W_n$ be the 4-manifold that is obtained from the 4-torus by applying a $(1, 0, n)$-torus surgery on $T_1$ along $m = x$ and a $(0, 1, 0)$-torus surgery on $T_3$ along $l = b$. There is a diffeomorphism\begin{equation}\label{Diffeomorphism 1}W_n\approx S^2\times T^2\end{equation}for every $n\in \Z$.

Let $Z_n$ be the 4-manifold that is obtained from the 4-torus by applying a $(1, 0, n)$-torus surgery on $T_1$ along $m = x$ and a surgery along the loop $l = b$ with respect to the product framing induced by $l\subset \nu(T_3)$. There is a diffeomorphism\begin{equation}\label{Diffeomorphism 2}Z_n\approx S^2\times T^2\# S^2\times S^2\end{equation}for every $n\in \Z$.

 Let $M_n$ be the 4-manifold that is obtained from the 4-torus by applying a $(0, 1, 1)$-torus surgery on $T_2$ along $m = y$, a $(1, 0, n)$-torus surgery on $T_1$ along $m = x$, and a $(0, 1, 0)$-torus surgery on $T_3$ along $l = b$. There is a diffeomorphism\begin{equation}\label{Diffeomorphism 5}M_n\approx S^1\times S^3\end{equation}for every $n\in \Z$.

Let $F_n$ be the 4-manifold that is obtained from the 4-torus by applying a $(0, 1, 1)$-torus surgery on $T_2$ along $m = y$, a $(1, 0, n)$-torus surgery on $T_1$ along $m = x$ and a surgery along the loop $l = b$ with respect to the product framing induced by $l\subset \nu(T_3)$. There is a diffeomorphism\begin{equation}\label{Diffeomorphism 6}F_n\approx S^1\times S^3\# S^2\times S^2\end{equation}for every $n\in \Z$.

\end{lemma}

\begin{proof} Any diffeomorphism $f:M^3\rightarrow M^3$ of a 3-manifold $M^3$ extends to a diffeomorphism $f\times \id:M^3\times S^1\rightarrow M^3\times S^1$. To show the existence of the diffeomorphism (\ref{Diffeomorphism 1}), after fixing an $n\in \Z$, we use a standard argument used by Baldridge-Kirk in \cite[Proof Lemma 2]{[BaldridgeKirk2]}. The 4-manifold $W_n$ is diffeomorphic to the product $M^3\times S^1 $ of a 3-manifold $M^3$ with the circle. The 3-manifold $M^3$ fibers over the circle with monodromy given by $D_x^n\circ D_b^0$, where $D_x$ and $D_b$ are Dehn twists along $x$ and $b$ in $T^2$. These Dehn twists correspond to an application of an $(1, n)$-Dehn surgery along $m = x$ and a $(1, 0)$-Dehn surgery along $l = b$ to the 3-torus $x\times y\times b$; cf. (\ref{4-torus}). It is immediate to see that $M^3$ is diffeomorphic to $S^2\times S^1$. The loop $y$ generates the fundamental group $\pi_1(M^3) = \Z y$. We conclude $W_n\approx (S^2\times S^1)\times S^1$. Since the choice of $n$ was arbitrary, the existence of the diffeomorphism (\ref{Diffeomorphism 1}) follows. Adapting the previous argument to the construction of $M_n$ allows us to conclude the existence of the diffeomorphism (\ref{Diffeomorphism 5}).

We argue the existence of the diffeomorphism (\ref{Diffeomorphism 2}) using a classical argument due to Moishezon \cite[Lemma 13]{[Moishezon]} (see the exposition in \cite[Lemma 3]{[Gompf2]} too). Fix an integer $n\in \Z$, and let $N_n = T^4_{T_1, b}(1, 0, n)$ be the 4-manifold that is obtained from the 4-torus by applying a $(1, 0, n)$-torus surgery on $T_1$ along $m = x$. Part of the discussion in the previous paragraph implies that $N_n \approx X_n\times S^1$ for a 3-manifold $X_n$ that fibers over the circle with monodromy given by the corresponding Dehn twist $D_x^n$. The 4-manifold $Z_n$ is obtained from $N_n$ by performing surgery along the essential loop $l = b\subset T_3$ with the framing that is induced by the product structure of $\nu(T_3)= T^2\times D^2$; the 4-manifold $Z_n$ corresponds to $N^\ast$ in Gompf's notation in \cite[proof of Lemma 3]{[Gompf2]}. The surgery along $l = b$ with the chosen framing does not change the intersection form over the integers nor the second Stiefel-Whitney class. In particular, $Z_n = \hat{N}_n\# S^2\times S^2$ for some 4-manifold $\hat{N}_n$. As Gompf explains \cite[Proof of Lemma 3]{[Gompf2]}, performing surgery to $N^\ast =Z_n$ on one of these 2-spheres yields the 4-manifold $\hat{N}_n$, and the aforementioned argument of Moishezon implies that $\hat{N}_n$ is obtained from $N_n$ by performing a $(0, 1, 0)$-torus surgery on $T_3$ along $l = b$, i.e.,\begin{equation}\hat{N}_n = (N_n)_{T_3, b}(0, 1, 0) = W_n \approx S^2\times T^2\end{equation}by (\ref{Diffeomorphism 1}). We conclude that the 4-manifold $Z_n$ is diffeomorphic to $S^2\times T^2\#S^2\times S^2$. Since $n$ was arbitrary, we conclude that there is a diffeomorphism (\ref{Diffeomorphism 2}) for every $n\in \Z$. Adapting the previous argument to the construction of $F_n$ allows us to conclude the existence of the diffeomorphism (\ref{Diffeomorphism 6}). This concludes the proof of Lemma \ref{Lemma Diffeomorphism 1}.

 \end{proof}

We record several properties of our principal building block in the proof of Theorem \ref{Theorem A}. 

\begin{lemma}\label{Lemma Building Block}Denote by $\{W_n(1): n\in \Z\}$ the infinite set of 4-manifolds that are obtained by applying a $(1, 0, n)$-torus surgery on $T_1$ along $m = x$ to the 4-torus for any $n\in \Z - \{0\}$; cf. Lemma \ref{Lemma Diffeomorphism 1}. Since the 2-tori $T_2, T_3\subset T^4$ are disjoint from the surgery, we conclude that $T_2, T_3\subset W_n(1)$ for every $n\in \Z$. These 2-tori and their geometric duals generate the second homology group $H_2(W_n(1); \Z) = \Z^4$ and the intersection form over the integers is\begin{equation}Q_{W_n(1)} = \overset{2}{\underset{i = 1}\bigoplus} \left( \begin{array}{cc}
0 & 1  \\
1 & 0 \end{array} \right).\end{equation} Moreover, the 4-manifold $W_n(1)$ is aspherical for every $n\in \Z - \{0\}$.
\end{lemma}

The claim regarding asphericity follows from a  modification to an argument of Baldridge-Kirk \cite[Lemma 2]{[BaldridgeKirk2]}: the 4-manifold $W_n(1)$ is aspherical since it is a product $W_n(1) = S^1\times U$ of the circle and an aspherical 3-manifold $U$ for every $n\in \Z - \{0\}$. 

We finish this section with a discussion of an application of work of Freedman-Quinn \cite{[FreedmanQuinn]} and a  homeomorphism criteria that is employed in the assembly of the 4-manifolds for the proof of Theorem \ref{Theorem A} in Section \ref{Section Proof of Theorem A}. 

\begin{lemma}\label{Lemma Intersection Form}Let $X$ be a closed connected simply connected 4-manifold that contains a smoothly embedded 2-torus $T\hookrightarrow X$ of self-intersection $[T]^2 = 0$ and with simply connected complement $\pi_1(X\setminus \nu(T)) = \{1\}$. For any $n\in \Z$, consider the fiber sum\begin{equation}\label{Desc}X_n(1):= (X\setminus \nu(T))\cup_{\phi} (W_n(1)\setminus \nu(T_2))\end{equation}where the gluing diffeomorphism\begin{equation}\label{Desc Glue}\phi:\partial (X\setminus \nu(T))\rightarrow \partial (W_n(1)\setminus\nu(T_2))\end{equation}maps the meridian $\mu_T$ in $X\setminus \nu(T)$ to the meridian $\mu_{T_2}$ in $W_n(1)\setminus \nu(T_2)$. The closed connected orientable 4-manifold $X_n(1)$ has infinite cyclic fundamental group, and its $\Lambda$-intersection form is extended from the integers. In particular, $X_n(1)$ is homeomorphic to $X\#S^2\times S^2\#S^1\times S^3$ for every $n \in \Z$.

\end{lemma}

A more precise description of the 4-manifold (\ref{Desc}) and the gluing diffeomorphism (\ref{Desc Glue}) that makes it clear that its Seiberg-Witten invariant is non-trivial is given in Section \ref{Section Proof of Theorem A}.

\begin{proof} The hypothesis $\pi_1(X) = \{1\} = \pi_1(X\setminus \nu(T))$ and the Seifert-van Kampen theorem imply that $\pi_1(X_n(1)) = \Z$ for every $n\in \Z$. If the second Stiefel-Whitney class of $X$ satisfies $w_2(X) = 0$, then $w_2(X(1)_n) = 0$ for every $n\in \Z$ by a result of Gompf \cite[Proposition 1.2]{[Gompf3]}. Fix an $n\in \Z$, and let $\Lambda = \Z[t^\pm]$. For 4-manifolds with infinite cyclic fundamental group, the homology $H_\ast(X_n(1); \Lambda)$ is computed as the homology of the universal cover as a $\Lambda$-module. Since $W_n(1)$ is aspherical, its universal cover of $W_n(1)$ is contractible, and we have that $H_2(\widetilde{X_n(1)}; \Z) = \Lambda^{b_2(X) + 2}$ \cite{[FreedmanQuinn]} and\begin{equation}H_2(X_n(1); \Lambda) = (H_2(X; \Z)\oplus \Z^2)\otimes_{\Z} \Lambda.\end{equation}

Using the $\Z$-valued intersection form of $W_n(1)$ of Lemma \ref{Lemma Building Block}, we compute the $\Lambda$-valued intersection form to be\begin{equation}\lambda_{X_n(1)} = \Big(Q_X\oplus \left( \begin{array}{cc}
0 & 1  \\
1 & 0 \end{array} \right)\Big)\otimes_{\Z} \Lambda.\end{equation}Since $n\in \Z$ was arbitrary, a result of Freedman-Quinn \cite[Theorem 107.A (2)]{[FreedmanQuinn]}, \cite{[StongWang]} implies that every member in the infinite set $\{X_{n}(1): n\in \Z\}$ is homeomorphic to the connected sum $X\#S^2\times S^2\#S^1\times S^3$.
\end{proof}

\subsection{Examples out of Reverse-Engineering of 4-manifolds}\label{Section RE}We now describe how Fintushel-Park-Stern's work can be used to produce examples of infinitely many smoothly non-isotopic nullhomologous 2-tori inside $S^2\times S^2$ and nullhomologous 2-spheres inside $S^2\times S^2\# S^2\times S^2$. In \cite[Section 4]{[FintushelParkStern]}, Fintushel-Park-Stern construct an infinite set $\{X_n: \Sw_{X_n} = n\in \Z\}$ of pairwise non-diffeomorphic closed smooth 4-manifold such that $X_n$ is homologically equivalent to $S^2\times S^2$ for every $n\in \Z$. We briefly recall their construction and set-up some notation. Let $\{a_1, b_1, a_2, b_2\}$ and $\{c_1, d_1, c_2, d_2\}$ be loops whose homotopy classes form a standard set of generators for the fundamental group $\pi_1(\Sigma_2)\times \pi_1(\Sigma_2) = \pi_1(\Sigma_2\times\Sigma_2)$, where $\Sigma_g$ is a closed orientable surface of genus two. Push-offs of a loop $a_i\subset \Sigma_2$ are denoted by $a_i'$ and $a_i''$. Once the 4-manifold $\Sigma_2\times \Sigma_2$ is equipped with the canonical product symplectic form, Fintushel-Park-Stern choose the following eight disjoint homologically essential Lagrangian 2-tori and surgery curves: 
\begin{center}
\begin{itemize}
\item $T_1: = a_1'\times c_1'$, $m_1 = a_1'$, $T_2: = a_2'\times c_1'$, $l_2 = c_1'$
\item $T_3: = a_2''\times d_1'$, $l_3 = d_1'$, $T_4: = b_1'\times c_1''$, $m_4 = b_1'$
\item $T_5: = a_2'\times c_2'$, $m_5 = a_2'$, $T_6: = a_1'\times c_2'$, $l_6 = c_2'$
\item $T_7: = a_1''\times d_2'$, $m_5 = d_2'$, and $T_8: = b_2'\times c_2''$, $m_1 = b_2'$.
\end{itemize}
\end{center}

Fintushel-Park-Stern construct an infinite set\begin{equation}\label{Homology 4-manifolds}\{X(1, n): SW_{X(1, n)} = n\in \Z\}\end{equation}of pairwise non-diffeomorphic closed smooth 4-manifolds that are homologically equivalent to the connected sum $2(S^2\times S^2)\# S^1\times S^3$ with infinite cyclic first homology group\begin{equation}\label{First Homology Generator}H_1(X(1, n); \Z) = \Z b_2\end{equation}by applying to $\Sigma_2\times \Sigma_2$ the following torus surgeries with respect to the Lagrangian framing 

\begin{center}
\begin{itemize}
\item  $(1, 0, -1)$-torus surgery on $T_1$ along $\gamma_1 = a_1'$,
\item $(0, 1, 1)$-torus surgery on $T_2$ along the surgery curve $\gamma_2 = c_1'$,
\item $(0, 1, 1)$-torus surgery on $T_3$ along the surgery curve $\gamma_3 = d_1''$,
\item $(1, 0 - 1)$-torus surgery on $T_4$ along the surgery curve $\gamma_4 = b_1'$,
\item $(1, 0 - 1)$-torus surgery on $T_5$ along the surgery curve $\gamma_5 = a_2'$,
\item $(0, 1, 1)$-torus surgery on $T_6$ along the surgery curve $\gamma_6 = c_2'$, and
\item $(0, 1, n)$-torus surgery on $T_7$ along the surgery curve $\gamma_7 = d_2'$.
\end{itemize}
\end{center}

For each $n\in \Z$, the corresponding 4-manifold contains the eighth 2-torus\begin{equation}\label{Eighth torus}T_8\subset X(1, n),\end{equation}which carries a loop whose homotopy/homology class corresponds to the generator $b_2$ in (\ref{First Homology Generator}). We build a 4-manifold $X(0, n)$ by applying a

$\bullet$ $(1, 0, 0)$-torus surgery to $X(1, n)$ on $T_8$ along the surgery curve $\gamma_8 = b_2'$.

Similarly, build a 4-manifold $Y(0, n)$ by applying 

$\bullet$ a surgery to $X(1, n)$ along the curve $\gamma_8 = b_2'$ using the 2-torus $T_8$ to induce the product framing on $\gamma_8$ as in Definition \ref{Definition Product Framing}. The following diffeomorphism types are obtained this way.

\begin{lemma}\label{Lemma FintushelParkStern}There are diffeomorphisms\begin{equation}\label{Diffeomorphism FPS 1}X(0, n)\approx S^2\times S^2\end{equation}and\begin{equation}\label{Diffeomorphism FPS 2}Y(0, n)\approx S^2\times S^2\# S^2\times S^2\end{equation}for every $n\in \Z$.
\end{lemma}

An argument to prove Lemma \ref{Lemma FintushelParkStern} can be found in \cite[Proposition 6]{[FintushelStern2]}.\begin{hproof}We outline an argument to prove Lemma \ref{Lemma FintushelParkStern} that is based on handlebody calculus and which builds on work of Akbulut in \cite{[Akbulut2], [Akbulut3], [Akbulut4]}. Akbulut draws a handlebody diagram of the handle decomposition of a genus two surface bundle over a genus two surface in \cite[Figure 2]{[Akbulut2]}, and the following small modification to it yields a handlebody of the product $\Sigma_2\times \Sigma_2$ of a pair of  surfaces of genus two. Deconstruct the latter into two copies of the product of a genus two surface and a punctured 2-torus\begin{equation}\label{E0}\Sigma_2\times (T^2\backslash  D^2)\end{equation}glued along the common $\Sigma_2 \times S^1$ boundary\begin{equation}\label{Product Manifold}\Sigma_2\times \Sigma_2 = ((\Sigma_2\times T^2)\backslash \Sigma_2\times D^2) \cup ((\Sigma_2\times T^2)\backslash \Sigma_2\times D^2).\end{equation}That is, (\ref{Product Manifold}) is the double of (\ref{E0}) and a handlebody of the latter is constructed in \cite[Section 3, Figure 11]{[Akbulut2]}. Akbulut calls this 4-manifold $E_0$ in his paper. In order to make the surgery 2-tori visible, Akbulut glues two copies of (\ref{E0}) using a cylinder $\Sigma_2\times S^1\times [0, 1]$ as\begin{equation}\label{Product Manifold 2}\Sigma_2\times \Sigma_2 = ((\Sigma_2\times T^2)\backslash \Sigma_2\times D^2) \cup_{\id}(\Sigma_2\times S^1\times [0, 1])\cup_{\id^{-1}} ((\Sigma_2\times T^2)\backslash \Sigma_2\times D^2).\end{equation}A handlebody diagram for the handle decomposition of (\ref{Product Manifold 2}) is drawn by substituting the lower part of \cite[Figure 23]{[Akbulut2]} with a copy of the upper part, i.e., two copies of (\ref{E0}). The eight torus surgeries that are performed to $\Sigma_2\times \Sigma_2$ to construct $X(0, n)$ as  indicated in Section \ref{Section RE} are divided in two sets of four torus surgeries. Each set is performed to a copy of (\ref{E0}) in the decomposition (\ref{Product Manifold 2}). Akbulut constructs a handlebody diagram of the handle decomposition of the 4-manifold $\tilde{E}_0$ that is obtained from (\ref{E0}) by applying four torus surgeries with coefficients $(1, 0, 1)$ or $(0, 1, 1)$ in \cite[Figure 6]{[Akbulut3]}; the coefficients of these torus surgeries correspond to Luttinger surgeries \cite{[AurouxDonaldsonKatzarkov], [Luttinger]}. If we were to glue two copies of the handlebody of $\tilde{E}_0$, we would obtain a handlebody of an irreducible symplectic 4-manifold that is homology equivalent to $S^2\times S^2$. Instead of doing so, we modify the handlebody of $\tilde{E}_0$ in \cite[Figure 6]{[Akbulut3]} to obtain a handlebody of the 4-manifold $\hat{E}_0$ that is obtained from (\ref{E0}) by applying two Luttinger surgeries, a $(0, 1, n)$-torus surgery, and a $(1, 0, 0)$-torus surgery; see Section \ref{Section RE}. This consist in changing the framing coefficients in \cite[Figure 6]{[Akbulut3]} of the 2-handles involved in the torus surgeries. A handlebody diagram of $X(0, n) = \tilde{E}_0\cup \hat{E}_0$ for fixed $n\in \Z$ is constructed from a copy of \cite[Figure 6]{[Akbulut3]} and its modified handlebody diagram. Several handle slides pivoted on the 0-framed 2-handles unlink the diagram, and handle cancellations allows us to conclude that $X(0, n)$ is diffeomorphic to $S^2\times S^2$. Surgering a circle-dot 1-handle in the handlebody diagram of $X(0, n)$ to a 0-framed 2-handle allows us to conclude that $Y(0, n)$ is diffeomorphic to $S^2\times S^2\#S^2\times S^2$.

\end{hproof}

\subsection{Fintushel-Stern's invariant of 2-tori}\label{Section SW Invariants} In this section we recall the invariant that we will use to distinguish the 2-tori of Theorem \ref{Theorem A} and follow almost verbatim the exposition in \cite[Section 2]{[FintushelSternTori]}. The Seiberg-Witten invariant of a smooth closed oriented 4-manifold $X$ is a map $\Sw_X':\mathcal{S}_X\rightarrow \Z$ from the set $\mathcal{S}_X$ of isomorphism classes of $\Spin^{\C}$-structures on a closed 4-manifold $X$ to the integers. We fix an orientation on $H^0(X; \Z)\otimes \det H^2_+(X; \R)\otimes \det H^1(X; \R)$, hence a sign of $\Sw_X'$ in what follows; see \cite[Remark 1.2]{[MorganMrowkaSzabo]}. Fintushel-Stern \cite{[FintushelSternTori]} use the modified Seiberg-Witten invariant\begin{equation}\Sw_X: \{k\in H_2(X; \Z): k = w_2(X) \mod 2\}\rightarrow \Z\end{equation}defined by\begin{equation}\Sw_X(k) = \sum_{c(\mathfrak{s}) = k}\Sw_X'(\mathfrak{s})\end{equation}where $c(\mathfrak{s})\in H_2(X; \Z)$ is the Poincar\'e dual to the first Chern class $c_1(W^+_{\mathfrak{s}})$ of the bundle of positive spinors $W^+_{\mathfrak{s}}$ over $X$ that corresponds to the $\Spin^\C$-structure $\mathfrak{s}$. Denote by $\{\pm \beta_1, \ldots, \pm \beta_n\}$ the set of basic classes of $X$ and regard the Seiberg-Witten invariant of $X$ as an element of the integral group ring $\Z H_2(X)$ in terms of the Laurent polynomial\begin{equation}\mathcal{SW}_X = \Sw_X(0) + \sum^n_{j = 1} \Sw_X(\beta_j)\cdot (t_{\beta_j} + (-1)^{(\chi + \sigma)/4)}t^{-1}_{\beta_i})\in \Z H_2(X)\end{equation}for $t_{\beta_j}$ the element in the group ring that corresponds to $\beta_j\in H_2(X)$, the Euler characteristic of $X$ is denoted by $\chi$ and the signature by $\sigma$.


The Seiberg-Witten invariant of a 4-manifold $X_{T, \gamma}(p, q, r)$ that is obtained by performing a torus surgery along $T\subset X$ as defined in Section \ref{Section Torus Surgeries} is calculated using the Morgan-Mrowka-Szab\'o formula \cite{[MorganMrowkaSzabo]} given by\begin{equation}\label{Equation 1}
        \begin{split}
        \sum_i\Sw_{X_{T, \gamma}(p, q, r)}(k_{(p, q, r)} + i[T_{p, q, r}]) =
         p\sum_i \Sw_{X_{T, \gamma}(1, 0, 0)}(k_{(1, 0, 0)}  + i [T_{(1, 0, 0)}]) + \\
        + q \sum_i \Sw_{X_{T, \gamma}(0, 1, 0)}(k_{(0, 1, 0)} + i [T_{(0, 1, 0)}]) + r \sum_i \Sw_{X_{T, \gamma}(0, 0, 1)}(k_{(0, 0, 1)} + i [T_{(0, 0, 1)}]).\ 
        \end{split}
    \end{equation}In (\ref{Equation 1}), we denote by $T_{(a, b, c)}$ the core 2-torus of $X_{T, \gamma}(a, b, c)$   for any $a, b, c\in \Z$ and $k_{(a,b,c)}\in H_2(X_{T, \gamma}(a,b,c))$ is any class that agrees with the restriction of a given class $k\in H_2(X)$ in $H_2(X\setminus \nu(T), \partial)$ in the diagram\begin{equation}\label{Equation 2}
\begin{CD}H_2(X_{T, \gamma}(p,q,r)) @>>> H_2(X_{T, \gamma}(p, q, r), T\times D^2)\\@.     @VV\cong V\\@.   H_2(X\setminus T\times D^2, \partial)\\ @.    @VV\cong V \\
H_2(X) @>>> H_2(X, T\times D^2).
\end{CD}
\end{equation}

There is an indeterminacy in the formula (\ref{Equation 1}) due to multiples of the core 2-tori $[T_{a, b, c}]$, and one removes it as follows. From the diagram (\ref{Equation 2}), we obtain both a map\begin{equation}\label{Composition Maps}\pi(a,b,c): H_2(X_{T, \gamma}(a,b,c))\rightarrow H_2(X\setminus T\times D^2, \partial)\end{equation}and the induced map on integral group rings\begin{equation}\label{Composition Maps2}\pi(a,b,c)_\ast : \Z H_2(X_{T, \gamma}(a,b,c))\rightarrow \Z H_2(X\setminus T\times D^2, \partial),\end{equation}and work with the invariant\begin{equation}\label{New Invariant}\overline{\mathcal{SW}}_{(X, T)} = \pi(a,b,c)_\ast(\mathcal{SW}_{X_{T, \gamma}}(a,b,c))\in \Z H_2(X\setminus T\times D^2, \partial).\end{equation}Using (\ref{New Invariant}), we rewrite formula (\ref{Equation 1}) as\begin{center}$\overline{\mathcal{SW}}_{(X_{T, \gamma}(p,q,r), T_{(p,q,r)})}  =$\end{center}\begin{center}$= p \cdot \overline{\mathcal{SW}}_{(X_{T, \gamma}(1,0,0), T_{(1,0,0)})} + q \cdot\overline{\mathcal{SW}}_{(X_{T, \gamma}(0,1,0), T_{(0,1,0)})}+ r \cdot \overline{\mathcal{SW}}_{(X_{T, \gamma}(0,0,1), T_{(0,0,1)})}$.\end{center}

Once the indeterminacy in (\ref{Equation 1}) has been removed, Fintushel-Stern used this collection of Seiberg-Witten invariants to define an invariant of the pair $(X, T)$ as follows.

\begin{definition}\label{Definition FS Invariant for 2-Tori}\cite[p. 951]{[FintushelSternTori]}. Let $T$ be an embedded 2-torus in a closed oriented 4-manifold $X$ such that $[T]^2 = 0$. The Fintushel-Stern invariant of the pair $(X, T)$ is\begin{equation}\mathcal{I}(X, T) := \{\overline{\mathcal{SW}}_{(X_{T, \gamma}(a,b,c), T_{(a,b,c)})} : a, b, c \in \Z\}.\end{equation}

\end{definition}

We finish this section with the following proposition, which will be used to discern the 2-tori of Theorem \ref{Theorem A}. 

\begin{proposition}\label{Proposition FS Invariant for 2-Tori}Fintushel-Stern \cite[Proposition 2.1]{[FintushelSternTori]}. Let $T_i\hookrightarrow X$ be an embedded nullhomolgous 2-tori for $i = 1, 2$ inside a 4-manifold $X$ with a fixed homology orientation, and with $b_2^+(X\setminus \nu(T_i)) > 1$. If $\mathcal{I}(X, T_1)\neq \mathcal{I}(X, T_2)$, then there is no diffeomorphism of pairs $(X, T_1)\rightarrow (X, T_2)$.

\end{proposition}

\section{Proof of main results.}\label{Section Proofs Main Results}


\subsection{Proof of Theorem \ref{Theorem A}}\label{Section Proof of Theorem A}We follow the notation of Section \ref{Section 2-tori in the 4-torus}. The proof consists of nine steps. The first step is to construct a closed 4-manifold $X(2)$ with four properties \footnote{Szab\'o describes a similar construction in \cite[Section 2]{[Szabo]}}:

(1) it has a symplectic structure \footnote{A result of Taubes \cite{[Taubes]} implies $\Sw_{X(2)}\neq 0$}, 

(2) its fundamental group is isomorphic to\begin{equation}\label{Fundamental group X2}\Z^2 = \Z x \oplus \Z b.\end{equation}

(3) There are two disjoint homologically essential  2-tori of self-intersection zero $T_1$ and $T_3$ embedded in $X(2)$ such that the inclusion induced homomorphism\begin{equation}\label{Isomorphism Group Complement 1}\pi_1(X(2)\backslash(\nu(T_1)\sqcup \nu(T_3)))\rightarrow \pi_1(X(2))\end{equation}is an isomorphism. The torus $T_1$ contains a loop whose homotopy class corresponds to $x$ and the torus $T_3$ contains a loop whose homotopy class corresponds to $b$, where $x$ and $b$ are the generators of the group (\ref{Fundamental group X2}).

(4) The Euler characteristic satisfies $\chi(X(2)) = \chi(X)$, the signature is $\sigma(X(2)) = \sigma(X)$, and the second Stiefel-Whitney class is $w_2(X(2)) = w_2(X)$. 

Define the 4-manifold as the symplectic sum \cite{[Gompf3]}\begin{equation}\label{Manifold First Step}X(2):= (X\backslash \nu(T)) \cup (T^2\times T^2\backslash \nu (T_2)),\end{equation}where $T_2 = Y\times A_2$ as in Section \ref{Section 2-tori in the 4-torus}. Notice that $T_2$ is a homologically essential Lagrangian submanifold of a symplectic 4-torus, and the symplectic structure can be perturbed so that $T_2$ becomes a symplectic submanifold \cite[Lemma 1.6]{[Gompf3]}. The 4-manifold (\ref{Manifold First Step}) admits a symplectic structure and this concludes the proof of Property (1). We now argue that (\ref{Manifold First Step}) satisfies Property (2). Let $\mu_2$ be the homotopy class of the meridian of the torus (\ref{Torus 2}). The hypothesis on the existence of an isomorphism $\pi_1(X\backslash \nu(T))\rightarrow \pi_1(X) = \{1\}$ implies that the relations\begin{equation}\label{Relations 1}\mu_2 = 1 = y = a\end{equation}hold in the fundamental group of $X(2)$ and we conclude that the group $\pi_1(X(2))$ is the rank two free abelian group on the generators $x$ and $b$ using the Seifert-van Kampen theorem. Regarding Property (3), notice that the 2-tori (\ref{Torus 1}) and (\ref{Torus 3}) are disjoint from the 2-torus (\ref{Torus 2}) that was employed in the construction of $X(2)$.  Therefore, both (\ref{Torus 1}) and (\ref{Torus 3}) are contained in $X(2)$. To show that the group isomorphism (\ref{Isomorphism Group Complement 1}) exists, we show that the meridians $\mu_1$ and $\mu_3$ of the two 2-tori are nullhomotopic in $X(2)\backslash(\nu(T_1)\sqcup \nu(T_2))$. Following the calculations of Baldridge-Kirk \cite[p. 922]{[BaldridgeKirk1]}, with our choice of framings for $T_1$ and $T_3$, we can conclude that the meridian of $T_1$ is given by $[\tilde{b}, \tilde{y}]$ and the meridian of $T_3$ is $[\tilde{x}, \tilde{y}]$ where $\tilde{g}$ denotes a conjugate of the element $g$. By (\ref{Relations 1}), we conclude that the relations\begin{equation}\label{Relations 2}\mu_1 = 1 = \mu_3\end{equation}hold in the group $\pi_1(X(2)\backslash(\nu(T_1)\sqcup \nu(T_2)))$ and the existence of the isomorphism (\ref{Isomorphism Group Complement 1}) follows. We now address Property (4) of (\ref{Manifold First Step}). A Mayer-Vietoris sequence reveals that $\chi(X(2)) = \chi(X)$. Novikov additivity implies $\sigma(X(2)) = \sigma(X)$. An argument guaranteeing that the second Stiefel-Whitney class of $X(2)$ is zero whenever $w_2(X) = 0$ was given by Gompf in \cite[Proposition 1.2]{[Gompf3]}.

The second step is to construct an infinite set\begin{equation}\label{Infinite Set 4-manifolds}\{X_n(1) : n\in \Z\}\end{equation} of closed irreducible 4-manifolds with infinite cyclic fundamental group $\Z b$ that satisfy three properties:

(1') $X_n(1)$ is homeomorphic to $X\#S^2\times S^2\# S^1\times S^3$ for every $n\in \Z$.

(2') $X_{n_1}(1)$ is not diffeomorphic to $X_{n_2}(1)$ if $n_1\neq n_2$.

(3') For every $n\in \Z$, there is a homologically essential 2-torus of self-intersection zero $T_3$ embedded in $X_n(1)$ such that the inclusion induced homomorphism\begin{equation}\label{Induced Isomorphism X1}\pi_1(X_n(1)\backslash(\nu(T_3)))\rightarrow \pi_1(X_n(1)) = \Z\end{equation}is an isomorphism. The torus $T_3$ contains a loop whose homotopy class $b$ is the generator of the fundamental group $\pi_1(X_n(1)) = \Z b$.

Fix an $n\in \Z$ and define a closed 4-manifold\begin{equation}\label{4-Manifold}X_n(1):= X(2)_{T_1, x}(1, 0, n),\end{equation}i.e., the 4-manifold that is obtained by applying a $(1, 0, n)$-torus surgery to $X$ on the 2-torus $T_1$ along the surgery curve $m = x$. Since the meridian $\mu_1$ is nullhomotopic in the complement (cf. (\ref{Isomorphism Group Complement 1}) and (\ref{Relations 2})), the Seifert-van Kampen theorem implies\begin{equation}\pi_1(X_n(1)) = \langle x, b : [x, b] = 1 = x\rangle = \Z b.\end{equation}The claim that the 4-manifold (\ref{4-Manifold}) satisfies Property (1') follows from Lemma \ref{Lemma Building Block} and Lemma \ref{Lemma Intersection Form}. Property (2') holds since the Seiberg-Witten invariants satisfy\begin{equation}\label{SW}\Sw_{X_{n_1}(1)}\neq \Sw_{X_{n_2}(1)}\end{equation} for $n_1\neq n_2$ as follows from the Morgan-Mrowka-Szab\'o formula (\ref{Equation 1}). For details on the computation of the Seiberg-Witten invariants of $X_n(1)$, the reader is directed to \cite[Section 3]{[Szabo]} or \cite[Corollary 1]{[FintushelParkStern]}. Results of Szab\'o \cite{[Szabo]} and Kotschick \cite[Section 5.2]{[Kotschick]} imply that the 4-manifolds (\ref{Infinite Set 4-manifolds}) are irreducible since $\Z$ is a residually finite group.  An argument to prove the claim that the 4-manifold (\ref{4-Manifold}) satisfies Property (3') is obtained by a small tweak to the argument used in the proof of Property (3) for $X(2)$. Indeed, the 2-torus $T_3\subset X(2)$ is disjoint from the surgery and is contained in $X(1)_n$. The existence of the isomorphism (\ref{Induced Isomorphism X1}) follows from (\ref{Relations 2}). 
This concludes the construction of the infinite set (\ref{Infinite Set 4-manifolds}) of pairwise non-diffeomorphic irreducible 4-manifolds in the homeomorphism class of $X\#S^2\times S^2\# S^1\times S^3$.

Let us set up the third step: fix $n\in \Z$ and apply a $(0, 1, 0)$-torus surgery to $X_n(1)$ on $T_2$ along $l = b$ and denote the resulting 4-manifold by $X_n$. That is,\begin{equation}\label{4-Manifold Final}X_n:= {X_n(1)}_{T_2, b}(0, 1, 0).\end{equation}In particular, $X_n$ is simply connected. The third step is to show that there is a diffeomorphism\begin{equation}\label{Final Diffeomorphism 1}X_n \approx X\end{equation}for every $n\in \Z$. The 4-manifold $X_n$ is diffeomorphic to a generalized fiber sum of $X$ and the 4-manifold that is obtained from the 4-torus by applying a $(1, 0, n)$-torus surgery on $T_1$ along $m = x$ and a $(0, 1, 0)$-torus surgery on $T_3$ along $l = b$ for a fixed $n\in \Z$. The latter is diffeomorphic to the product of a 2-sphere and 2-torus by  the diffeomorphism (\ref{Diffeomorphism 1}) of Lemma \ref{Lemma Diffeomorphism 1}. Notice that $X_1$ is the 4-manifold obtained as the symplectic sum of $X$ and $T^2\times S^2$ along $T$ and $T^2\times \{s\}\subset T^2\times S^2$. In particular, $X_n$ is obtained from $X$ by performing a $(0, 0, 1)$-torus surgery along $T$. This torus surgery changes nothing and produces the original 4-manifold $X$. Since the choice of $n$ was arbitrary, we conclude that there is a diffeomorphism (\ref{Final Diffeomorphism 1}).

The fourth step is to construct an infinite set of nullhomologous 2-tori as in (\ref{Inequivalent Tori}). Let\begin{equation}\label{Tori Built}T_n':= T^2\times \{0\}\subset X_n\end{equation}be the core 2-torus of the surgery (\ref{4-Manifold Final}). Notice that $T_n'$ is a nullhomologous 2-torus since $T_2$ was a homologically essential 2-torus for every $n\in \Z$. The infinite set (\ref{Inequivalent Tori}) is made of the image of (\ref{Tori Built}) under the diffeomorphism (\ref{Final Diffeomorphism 1}).

The fifth step is to argue that these 2-tori are pairwise topologically isotopic and inequivalent. A result of Sunukjian \cite[Theorem 7.2]{[Sunukjian]} says that the 2-tori are topologically unknotted since $b_2(X)\geq |\sigma(X)| + 6$. We use the invariant of Definition \ref{Definition FS Invariant for 2-Tori} to show that these 2-tori are smoothly inequivalent. We can undo the construction (\ref{4-Manifold Final}) by applying a torus surgery to $X_n$ on $T_n'$ and obtain $X_n(1)$ back for any $n\in \Z$. Since the set (\ref{Infinite Set 4-manifolds}) consists of 4-manifolds with pairwise different Seiberg-Witten invariant, possibly after passing to a subsequence, we have that $\mathcal{I}(X_{n_i}, T_{n_i}')\neq \mathcal{I}(X_{n_j}, T_{n_j}')$ for $i\neq j$\footnote{We could have used a similar argument to \cite[Section 3, Proof of Theorem 1.1]{[HoffmanSunukjian]} as well}. Hence, any two 2-tori (\ref{Tori Built}) are  inequivalent.

Let us set up the sixth step. Fix an $n\in \Z$ and consider the essential loop $l = b\subset T_3 \subset X_n(1)$ that generates the infinite cyclic group $\pi_1 X_n(1) = \Z$. We choose the framing on $l = b$ that is induced by the product structure $\nu(T_3) \approx T^2\times D^2$; see Section \ref{Section Spherical Modifications}. We do surgery along $l = b$ with this choice of framing to obtain a 4-manifold\begin{equation}\label{Final Spherical Modification}Y_n:= (X_n(1)\backslash \nu(b)) \cup (D^2\times S^2).\end{equation} The sixth step consists of showing that there is a diffeomorphism\begin{equation}\label{Final Diffeomorphism 2}Y_n \approx X\#S^2\times S^2.\end{equation} for every $n\in \Z$. This follows from  Lemma \ref{Lemma Diffeomorphism 1}, the diffeomorphism (\ref{Final Diffeomorphism 1}) and the argument due to Moishezon used in the proof of Lemma \ref{Lemma Diffeomorphism 1}. 

The seventh step consists of producing the nullhomotopic 2-spheres of (\ref{Inequivalent Spheres}). The belt 2-sphere\begin{equation}\label{Sphere Built}S_n:= \{0\}\times S^2\subset D^2\times S^2\subset Y_n\end{equation}is nullhomologous for every $n\in \Z$. Since $\pi_1Y_n = \{1\}$, we have that $S_n$ is nullhomotopic since $H_2(Y_n) = \pi_2(Y_n)$ for every $n\in \Z$ by the Hurewicz theorem. The infinite set (\ref{Inequivalent Spheres}) consists of the image of (\ref{Sphere Built}) under the diffeomorphism (\ref{Final Diffeomorphism 2}) for $n\in \Z$. This concludes the construction of the infinite set (\ref{Inequivalent Spheres}) embedded in $X\#S^2\times S^2$.

We discern these submanifolds in the eighth step by reversing the surgery (\ref{Final Spherical Modification}). That is, we perform surgery on the belt 2-sphere (\ref{Sphere Built}) to $Y_n \approx X\# S^2\times S^2$ and obtain $X(1)_n$ back. Given that two 4-manifolds $X(1)_{n_1}$ and $X(1)_{n_2}$ in the set (\ref{Infinite Set 4-manifolds}) are not diffeomorphic if $n_1\neq n_2$, there is no diffeomorphism of pairs $(X\#S^2\times S^2, S_{n_1})\rightarrow (X\#S^2\times S^2, S_{n_2})$. We conclude that the 2-spheres in the set (\ref{Inequivalent Spheres}) are pairwise inequivalent. 

In the ninth and final step, we prove that any two such 2-spheres are topologically unknotted. Theorem \ref{Theorem Homeomorphism Spheres} and the diffeomorphism (\ref{Final Diffeomorphism 2}) imply that there is a homeomorphism of pairs\begin{equation}\label{Homeomorphism A}F: (X\#S^2\times S^2, S_1)\rightarrow (X\#S^2\times S^2, S_2),\end{equation}which as a self-homeomorphism of $X\#S^2\times S^2$ has trivial normal invariant. Work of Quinn \cite{[Quinn]} says that $F$ is homotopic to the identity map; cf. \cite{[CochranHabegger]}. Work of Perron \cite{[Perron]} and Quinn \cite{[Quinn]} imply that the homeomorphism $F$ is topologically isotopic to the identity, and the isotopy takes $S_2$ onto $S_1$. Without loss of generality, we can assume that the exterior of $S_1$ is homeomorphic to $X\#S^2\times S^2\# S^1\times D^3$. In particular, $S_1$ bounds a locally flat embedded solid handlebody in $X\#S^2\times S^2$. It is also possible to invoke \cite[Theorem 7.2]{[Sunukjian]} to prove the claim. This concludes the proof of the theorem.

\hfill $\square$


\begin{remark}\label{Remark Smaller Topology} Infinite sets of pairwise inequivalent nullhomotopic 2-spheres embedded in $2\mathbb{CP}^2\# (k +1)\overline{\mathbb{CP}^2}$ and inequivalent 2-tori embedded in $\mathbb{CP}^2\#k\overline{\mathbb{CP}^2}$ are built by applying the techniques in this paper to work of Akhmedov-Park \cite{[AkhmedovPark1], [AkhmedovPark2]}, Baldridge-Kirk \cite{[BaldridgeKirk1], [BaldridgeKirk2]}, Fintushel-Stern \cite{[FintushelStern2], [FintushelStern3]}. Work of Akbulut on handlebodies \cite{[Akbulut2], [Akbulut3], [Akbulut4], [Akbulut5]}, work of Baykur-Sunukjian on stabilizations \cite{[BaykurSunukjian]} along with results of Moishezon \cite{[Moishezon]} and Gompf \cite{[Gompf2]} are useful to pin down the diffeomorphism type of the 4-manifolds constructed.
\end{remark}

\subsection{Proof of Theorem \ref{Theorem FHMT}}\label{Section FHMT} A closed non-smoothable topological 4-manifold $M$ with $\pi_1M =\Z$ that is not homotopy equivalent to a connected sum $M_0\#S^1\times S^3$, where $M_0$ is a simply connected 4-manifold was constructed in  \cite[Corollary 3]{[HambletonTeichner]}, \cite[Theorem 1.2]{[FriedlHambletonMelvinTeichner]}. Its Euler characteristic and signature are $\chi(M) = 4 = \sigma(M)$, its second Stiefel-Whitney class is non-zero $w_2(M)\neq 0$, and its Kirby-Siebenmann invariant is trivial $\Ks(M) = 0$. We perform the cut-and-paste operation of Section \ref{Section Spherical Modifications} solely in the topological category in this section. Construct a closed simply connected 4-manifold $Y:= M\backslash \nu(\gamma)\cup D^2\times S^2$, where an arbitrary framing has been chosen and $\gamma\subset M$ is a loop whose homotopy class generates $\pi_1M = \Z$. Although we have not specified a framing, we have that $w_2(Y)\neq 0$ and we can pin down a specific topological 4-manifold $Y$. Consider the nullhomotopic 2-sphere $\Sigma := \{0\}\times S^2\subset D^2\times S^2$ that is locally flat embedded in $Y$. A result of Freedman-Quinn \cite[Section 10.1]{[FreedmanQuinn]} implies that $Y$ is homeomorphic to $\# 4 \mathbb{CP}^2$. There is a locally flat nullhomologous 2-sphere $S \subset 4\mathbb{CP}^2$ whose exterior is $4\mathbb{CP}^2\# S^1\times D^3$. These two surfaces are concordant by \cite[Theorem 6.1]{[Sunukjian]}. The exteriors of $\Sigma$ and $S$ are not homotopy equivalent since $M$ is not homotopy equivalent to $4\mathbb{CP}^2\#S^1\times S^3$ \cite[Corollary 3]{[HambletonTeichner]}.

\hfill $\square$

\subsection{Proof of Theorem \ref{Theorem D}}\label{Section Further Examples} 




The constructions of the 4-manifolds and the surfaces (\ref{Tori Example B}) and (\ref{Knots Example B}) were described in Section \ref{Section 2-tori in the 4-torus}. The invariant of Definition \ref{Definition FS Invariant for 2-Tori} discerns these submanifolds as it was argued in Section \ref{Section Proof of Theorem A}. The corresponding diffeomorphisms are established in  (\ref{Diffeomorphism 1}), (\ref{Diffeomorphism 2}), (\ref{Diffeomorphism 5}), and (\ref{Diffeomorphism 6}) of Lemma \ref{Lemma Diffeomorphism 1}, and Lemma \ref{Lemma FintushelParkStern}.

\hfill $\square$

\end{document}